\documentclass[11pt]{article}
\usepackage[margin=1.1in]{geometry}
\usepackage{amsmath,amssymb,amsthm}
\usepackage{booktabs}
\usepackage[hidelinks]{hyperref}
\hypersetup{pdftitle={Straight-line programs for the solutions of Pell's equation},
pdfauthor={Bogdan Dumitru and Mihai Prunescu}}
\newtheorem{theorem}{Theorem}
\newtheorem{lemma}[theorem]{Lemma}
\newtheorem{proposition}[theorem]{Proposition}

\theoremstyle{remark}

\newcommand{\N}{\mathbb N}
\newcommand{\Z}{\mathbb Z}

\newcommand{\eps}{\varepsilon}
\newcommand{\tsub}{\mathbin{\dot{-}}}
\newcommand{\md}{\mathbin{\mathrm{mod}}}
\newcommand{\fl}[1]{\left\lfloor #1\right\rfloor}
\newcommand{\HW}{\operatorname{HW}}

\newcommand{\Gpair}{\mathsf{G}_{02}}
\newcommand{\Gtriple}{\mathsf{G}_{024}}
\newcommand{\Recover}{\mathsf{R}}
\newcommand{\Coord}{\mathsf{C}}
\newcommand{\Signed}{\mathsf{S}}
\newcommand{\Squared}{\mathsf{T}}

\begin{document}
\title{Straight-line programs for the solutions of Pell's equation}
\author{Bogdan Dumitru\footnote{University of Bucharest, Academiei 14,
Bucharest (RO-010014), Romania. Simion Stoilow Institute of Mathematics
of the Romanian Academy, Research unit 5, P. O. Box 1-764,
Bucharest (RO-014700), Romania.}
\and Mihai Prunescu\footnote{Research Center for Logic, Optimization and
Security (LOS), Faculty of Mathematics and Computer Science, University of
Bucharest, Academiei 14, Bucharest (RO-010014), Romania.}\footnote{Simion
Stoilow Institute of Mathematics of the Romanian Academy, Research unit 5,
P. O. Box 1-764, Bucharest (RO-014700), Romania.
E-mail: \texttt{mihai.prunescu@imar.ro}.}}
\date{}
\maketitle

\begin{abstract}
For nonsquare $d\ge2$, we construct fixed straight-line programs for
the least non-trivial solution $(X_1,Y_1)$ of $x^2-dy^2=1$, using addition,
truncated subtraction, multiplication, integer division, exponentiation,
and remainder. A geometric-sum identity recovers $(X_1,Y_1)$ from the
coordinate sums of an initial segment of Pell solutions, without knowing
how many solutions were summed. Weighted binary encodings make these sums
accessible to arithmetic computation. Counting operations with reuse of
computed values, one program uses $94$ operations, with intermediate
bit lengths bounded by $2^{2^{2^{O(d)}}}$. Four additional operations
give a $98$-operation program with the bound $2^{2^{O(d)}}$.
Hua's bound gives corresponding programs with $103$ and $107$ operations
and replaces $O(d)$ by $O(\sqrt d\log d)$ in these size bounds.
Once the fundamental solution is known, generating-function formulas
compute the $n$-th positive solution in $21$ further operations by
extracting its coordinates as base-$b$ digits. We prove that
$2X_1(X_1+1)-1$ is the least integer base for which both formulas
are defined and correct for every $n\ge1$.
\end{abstract}


\section{Introduction}\label{sec:intro}

Let $d\ge2$ be a nonsquare integer. The fundamental solution of
\[
  x^2-dy^2=1
\]
is the solution $(X_1,Y_1)$ with $X_1,Y_1>0$ and $X_1$ least.
Writing $\eps=X_1+Y_1\sqrt d$, the nonnegative solutions are
\[
  X_n+Y_n\sqrt d=\eps^n,\qquad n\ge0,
  \qquad (X_0,Y_0)=(1,0).
\]
Continued fractions find $(X_1,Y_1)$ by an iteration whose length depends
on $d$. We seek fixed finite arithmetic expressions in $d$ for this pair
and in $d,n$ for the other solutions. The problem is to give explicit
constructions with a small number of operations and to determine how
large their intermediate integers become.

Our first result recovers $(X_1,Y_1)$ from sums that do not identify any
individual solution. Choose $K>X_1$ and sum the
coordinates of all solutions in $0\le x,y<K$:
\[
  A=\sum_{n=0}^N X_n,\qquad B=\sum_{n=0}^N Y_n,
  \qquad X_N<K\le X_{N+1}.
\]
Although the endpoint $N$ is unknown, Theorem~\ref{lem:identity} gives
\[
  \eps=\frac{c+B\sqrt d}{c-B\sqrt d},
  \qquad c=A^2-A-dB^2.
\]
The ratio of the geometric sum $A+B\sqrt d$ to its conjugate gives
$\eps^N$, while removing the constant term from both sums gives
$\eps^{N+1}$.
Taking their quotient eliminates $N$. Thus a method that computes sums
over a whole square can recover its least positive Pell solution without
first isolating that solution. Section~\ref{sec:reconstruction} also
shows that either coordinate sum determines the other.

A second result concerns all the positive solutions once $(X_1,Y_1)$
is known. The generating-function formulas of Prunescu and
Sauras-Altuzarra \cite{PSA25} extract $X_n,Y_n$ as digits in a sufficiently
large integer base. Theorem~\ref{thm:general-base} determines the least
base that works for both coordinates at every $n\ge1$:
\[
  b_0=2X_1(X_1+1)-1.
\]
The obstruction already occurs at $n=1$. At a smaller base with positive
denominator, the two extraction errors cannot both disappear on reduction
modulo the base, because $X_1$ and $Y_1$ are coprime.
We prove this sharp bound in the final section and use it to construct
the general solution in $21$ operations after $(X_1,Y_1)$ is available.

An \emph{arithmetic term} is a finite composition, on
$\N=\{0,1,2,\ldots\}$, of
\[
  u+v,\qquad u\tsub v=\max(u-v,0),\qquad uv,
  \qquad \fl{u/v},\qquad u^v,
\]
with positive divisors. We also use
$u\md v=u\tsub v\fl{u/v}$.
Our counting convention charges one operation for remainder, as well as
one for each of the five displayed operations. Thus the reported counts
refer to this six-operation basis. A \emph{straight-line program} is a
finite sequence of assignments using these operations, with no loop or
input-dependent branch. Constants, assignments, and reuse of previously
computed values have no cost. All counts below measure the size of these
programs, or equivalently arithmetic circuits with shared intermediate
values. Substituting every assignment into its later uses produces
arithmetic terms, but repeats subexpressions and can substantially
increase their size as trees.

Existence of arithmetic terms for the Pell coordinates already follows
from Mazzanti's theorem \cite{Mazzanti}, which represents every
Kalm\'ar-elementary function in the five-operation basis above.
Indeed, the bound $X_1<64^d$ proved in
Proposition~\ref{prop:elementary} makes the fundamental solution a bounded
search for the least positive pair satisfying $x^2=1+dy^2$.
Extending the coordinates by zero on square inputs gives elementary
functions on $\N$. A bounded search below $(2\cdot64^d)^n$ does the same
for the $n$-th solution. Our contribution is to display short programs,
prove their correctness, and quantify both their operation counts and
their intermediate sizes.

Arithmetic zero counting encodes values of a polynomial as consecutive
groups of binary digits and reads the number of its zeros from the
Hamming weight, the number of ones in a binary expansion. This method is
developed for algebraic sets in \cite{Pru26} and for elliptic curves in
\cite{DP26}. To obtain coordinate sums instead of counts, we repeat the
digit belonging to a cell $(x,y)$ according to its weight. A Pell solution
then contributes $x$, $y$, or $x+2Ky$, according to the sum we want to
recover. The signed residual $F(x,y)=x^2-1-dy^2$ needs moments of degrees
zero and two, namely $\sum_{j<K}Q^j$ and $\sum_{j<K}j^2Q^j$ at suitable
integer bases $Q$. A negative residual contributes an additional binary
weight $\nu_2(-F)$, where $\nu_2$ is the exponent of two dividing a
positive integer. Section~\ref{sec:counting} bounds the sum of these
contributions by counting square roots modulo powers of two. Sufficiently
wide groups of digits allow integer division to remove the error.
Squaring the residual removes all negative values and hence this error.
It requires fourth moments but permits narrower groups of digits.

The five programs in Section~\ref{sec:programs} implement these choices.
They either encode the two sums together, encode $A$ and recover $B$, or
encode the two sums separately. Their arithmetic subroutines are given
in Section~\ref{sec:arithmetic}, including a term for Hamming weight.
Their packed integers differ in size, but the arithmetic expansion of
Hamming weight dominates the full size bounds. With elementary parameters,
the shortest program has $94$ operations and a triply exponential bound
on intermediate bit lengths. Choosing a smaller parameter inside the
Hamming-weight subroutine costs four operations and makes the bound doubly
exponential. Hua's bound \cite{Hua42,Len02} gives the corresponding
$103$- and $107$-operation programs with $\sqrt d\log d$ in place of $d$
in the innermost exponent. Section~\ref{sec:parameters} makes these
comparisons explicit. Exact checks in Section~\ref{sec:verification}
support the formulas and counts, and Section~\ref{sec:questions} discusses
the remaining question of reducing intermediate sizes.

\section{Reconstruction from coordinate sums}\label{sec:reconstruction}

Throughout the paper, $K$ is an integer with $K>X_1$. Since $Y_n<X_n$,
the solutions in $\mathcal B_K=\{0,\ldots,K-1\}^2$ form the initial
segment indexed by $0\le n\le N$, with $N\ge1$.
The sums $A,B$ therefore combine an unknown number of powers of the same
unit. To recover that unit, we compare each geometric sum with its
conjugate, first including the constant term and then omitting it.
The two ratios differ by one power of $\eps$, which makes their quotient
independent of the endpoint $N$.

\begin{theorem}[Reconstruction from coordinate sums]\label{lem:identity}
Put $c=A^2-A-dB^2$. Then $c>B\sqrt d>0$,
\[
  \eps=\frac{c+B\sqrt d}{c-B\sqrt d},
\]
and
\begin{equation}\label{eq:fundamental}
  X_1=\frac{c^2+dB^2}{c^2-dB^2},\qquad
  Y_1=\frac{2Bc}{c^2-dB^2}.
\end{equation}
Both quotients are integers.
\end{theorem}

\begin{proof}
Let $S=A+B\sqrt d=\sum_{n=0}^N\eps^n$, and let $\overline S$ be its
conjugate. The finite geometric sums give
\[
  \frac{S-1}{\overline S-1}=\eps^{N+1},\qquad
  \frac S{\overline S}=\eps^N.
\]
Their quotient is
\[
  \eps=\frac{(S-1)\overline S}{(\overline S-1)S}
       =\frac{c+B\sqrt d}{c-B\sqrt d}.
\]
Because $S>1$ and $\overline S>1$, the numerator and denominator are
positive. Also $B\ge Y_1\ge1$, so $c>B\sqrt d>0$.
Rationalizing the denominator and comparing the coefficients of
$1,\sqrt d$ proves \eqref{eq:fundamental}.
\end{proof}

Here is the subroutine $\Recover(d,A,B)$ used at the end of each
fundamental-solution program:
\begin{align*}
  V_B&=d(BB),& c&=(AA\tsub A)\tsub V_B,\\
  V_c&=cc,& D_c&=V_c\tsub V_B,\\
  X_1&=\fl{(V_c+V_B)/D_c},&
  Y_1&=\fl{(2B)c/D_c}.
\end{align*}
\emph{Operation count:} $2+3+1+1+2+3=12$.
All subtractions are exact, and $D_c>0$ by Theorem~\ref{lem:identity}.

\begin{lemma}\label{lem:sum-bound}
For $n\ge1$, one has $X_{n+1}\ge3X_n$. Consequently,
\[
  B<A<\frac32K<2K.
\]
\end{lemma}

\begin{proof}
The recurrence $X_{n+1}=2X_1X_n-X_{n-1}$ follows from
$\eps+\eps^{-1}=2X_1$.
Since $X_1\ge2$, we have $X_2=2X_1^2-1\ge3X_1$.
If $X_n\ge3X_{n-1}$, then
$X_{n+1}\ge4X_n-X_n/3>3X_n$.
Summing the resulting geometric bound backwards from $X_N$ gives
\[
  A=1+\sum_{n=1}^N X_n
   <1+\frac32X_N
   \le1+\frac32(K-1)<\frac32K.
\]
Since $Y_n<X_n$ for all $n\ge0$, summing gives $B<A$.
\end{proof}

The bound $A<2K$ lets us store both sums in $A+2KB$ and recover them by
division and remainder. We can also avoid storing a second coordinate.
Conjugation gives $A-\sqrt d B$ as a short decreasing geometric sum,
so $A/\sqrt d$ lies just above the integer $B$. An approximation to
$\sqrt d$ with a sufficiently small error preserves the integer part.
The following theorem supplies such a rational approximation and shows,
in both directions, that either coordinate sum determines the other.

\begin{theorem}\label{thm:coordinate}
For $j\ge1$, define
\[
  R_j=\frac{4^j}{\binom{2j}{j}},\qquad \omega=\frac{R_{dK}}{R_K}.
\]
Then
\begin{equation}\label{eq:coordinate}
  A=\fl{\omega B}+2,\qquad B=\fl{A/\omega}.
\end{equation}
\end{theorem}

\begin{proof}
The conjugate geometric sum satisfies
\[
  A-\sqrt d B=1+\tau,\qquad
  0<\tau<\frac1{\eps-1}<\frac12,
\]
because $\eps=X_1+\sqrt{X_1^2-1}\ge2+\sqrt3$.
Lemma~\ref{lem:sum-bound} also gives $B<3K/(2\sqrt d)$.

The recurrence $R_{j+1}/R_j=2(j+1)/(2j+1)$ shows that
$R_j^2/j$ strictly decreases. The sequence $R_j^2/(j+2/7)$
is nondecreasing and strictly increases for $j>1$, because
\[
  4(j+1)^2(j+2/7)-(2j+1)^2(j+1+2/7)=(j-1)/7.
\]
As $K\ge3$, comparison at $K$ and $dK$ yields
\begin{equation}\label{eq:ratio-bound}
  \frac{7dK+2}{7K+2}<\omega^2<d,\qquad
  0<d-\omega^2<\frac{2(d-1)}{7K+2}.
\end{equation}
In particular, $\omega>\sqrt d/2$. With $\Delta=\sqrt d-\omega$, we obtain
\[
  0<\Delta B=\frac{(d-\omega^2)B}{\sqrt d+\omega}
  <\frac{2(d-1)}{7d}<\frac12.
\]
Thus $\omega B=A-1-(\tau+\Delta B)$, with $0<\tau+\Delta B<1$,
and the first identity follows.

For the second identity, we prove $0<A-\omega B<\omega$.
Since $K\ge3$, \eqref{eq:ratio-bound} gives $\omega^2>(21d+2)/23$.
If $d\ge3$, then $\tau<(\sqrt3-1)/2<3/8$, and hence
\[
  0<A-\omega B<1+\frac38+\frac27=\frac{93}{56}<\frac53<\omega,
\]
where $\omega^2>65/23>25/9$ proves the last inequality.
If $d=2$, then $\eps=3+2\sqrt2$ and
$\tau<(\sqrt2-1)/2<3/14$. In this case
\[
  0<A-\omega B<1+\frac3{14}+\frac17=\frac{19}{14}<\frac{11}{8}<\omega,
\]
using $\omega^2>44/23>121/64$.
Dividing by $\omega$ gives $B<A/\omega<B+1$.
\end{proof}

The programs below use $B=\fl{A/\omega}$. The reverse identity explains
why the choice of which coordinate to encode is not essential.
To implement the recovery, we extract the binomial coefficients from
integer powers. Suppose $p=2^{\rho K}$ and $v=p^K$ are supplied, with
$\rho\ge2d$. The subroutine $\Coord(d,K,A,p,v)$ is
\begin{align*}
  h&=2K,& z_c&=p+1,& c_K&=\fl{z_c^h/v}\md p,\\
  H_c&=dh,& c_{dK}&=\fl{z_c^{H_c}/v^d}\md p,&
  a_c&=2^{H_c\tsub h}c_K,\\
  B&=\fl{A c_{dK}/a_c}.&&&&
\end{align*}
\emph{Operation count:} $1+1+3+1+4+3+2=15$.

All coefficients of $(p+1)^{2dK}$ are less than $p$, since they are
less than $2^{2dK}\le p$. Thus
$c_K=\binom{2K}{K}$ and $c_{dK}=\binom{2dK}{dK}$.
It follows that $a_c/c_{dK}=\omega$, so Theorem~\ref{thm:coordinate}
proves the output formula. The subtraction $H_c\tsub h$ is exact.

\section{Weighted counting of Pell solutions}\label{sec:counting}

To compute $A$, a solution $(x,y)$ must contribute its coordinate $x$
to the count. We arrange this by assigning a digit to each cell of the
square and repeating it $x$ times. If the digit has $w$ binary ones at a
nonzero residual and $2w$ at a zero, these repetitions give a known
background plus $wA$. The weights $y$ and $x+2Ky$ give $B$ and $A+2KB$
in the same way. We call the resulting integer a \emph{packed integer},
because its digits occupy separate groups of binary positions.

The obstacle is that $F(x,y)=x^2-1-dy^2$ can be negative. Subtracting a
negative value changes the pattern of binary carries and adds a valuation
to the digit's Hamming weight. We first identify that error, then bound
its sum so that integer division removes it.

\subsection{Digits for signed values}

\begin{lemma}\label{lem:digits}
Let $w\ge1$, $P=2^w$, and
$\delta_w(z)=(P-1)(P+1-z)$.
For $0\le z\le P$, this is a base-$P^2$ digit and
\[
  \HW(\delta_w(z))=
  \begin{cases}2w,&z=0,\\w,&z>0.\end{cases}
\]
For $|z|<P$, it is a positive base-$P^3$ digit and
\begin{equation}\label{eq:signed-digit}
  \HW(\delta_w(z))=
  \begin{cases}
    2w,&z=0,\\
    w,&z>0,\\
    w+\nu_2(-z),&z<0.
  \end{cases}
\end{equation}
\end{lemma}

\begin{proof}
At $z=0$, the digit is $P^2-1$.
For $1\le z\le P$, write
\[
  \delta_w(z)=(P-z)P+(z-1).
\]
The two coefficients occupy separate groups of $w$ bits and add to
$P-1$. Their binary expansions are complementary, so their Hamming
weights add to $w$.
If $z=-a$ with $1\le a<P$, then
\[
  \delta_w(-a)=P^2+(a-1)P+(P-1-a).
\]
These three binary pieces are disjoint. Since $P-1-a$ is the $w$-bit
complement of $a$, and
$\HW(a-1)-\HW(a)=\nu_2(a)-1$, their total weight is
$w+\nu_2(a)$. The bound $0<\delta_w(z)<2P^2\le P^3$ for $|z|<P$
ensures that each value fits in one base-$P^3$ digit.
\end{proof}

For example, take $w=3$, so $P=8$.
Then $\delta_w(0)=63=(111111)_2$ has six ones, whereas
$\delta_w(2)=49=(110001)_2$ has three.
The negative value $-4$ gives
$\delta_w(-4)=91=(1011011)_2$, with $3+\nu_2(4)=5$ ones.
Using groups of $3w$ bits keeps all three kinds of digit disjoint.

\begin{lemma}[weighted signed counting]\label{lem:signed-count}
Let $z_j$ be finitely many integers with $|z_j|<2^w$, and let
$\lambda_j$ be nonnegative integer weights. Form $M$ by placing
$\lambda_j$ copies of $\delta_w(z_j)$ in distinct base-$2^{3w}$
positions. Put
\[
  W=\sum_j\lambda_j,\qquad Z=\sum_{z_j=0}\lambda_j,
  \qquad \eta=\sum_{z_j<0}\lambda_j\nu_2(-z_j).
\]
Then $\HW(M)=w(W+Z)+\eta$. If $\eta<w$, then
\begin{equation}\label{eq:signed-count}
  \fl{\HW(M)/w}\tsub W=Z.
\end{equation}
For nonnegative $z_j\le2^w$, base $2^{2w}$ suffices and $\eta=0$.
\end{lemma}

\begin{proof}
Distinct positions in a power-of-two base have disjoint binary digits.
Sum the identities in Lemma~\ref{lem:digits}, with multiplicities
$\lambda_j$. If $0\le\eta<w$, integer division by $w$ removes $\eta$.
\end{proof}

\subsection{A bound for the contribution of negative Pell residuals}

The error from one negative residual can be as large as its binary
length, but large valuations occur only in cells satisfying a congruence
modulo a high power of two. Counting these cells at each power bounds
the total error and lets us choose the width before evaluating the packing.

\begin{lemma}\label{lem:valuation}
Suppose $K\ge64$ and $2\le d<K^2$. Then $|F(x,y)|<K^4$ on
$\mathcal B_K$, and
\begin{equation}\label{eq:valuation}
  \sum_{F(x,y)<0}\nu_2(-F(x,y))<7K^2.
\end{equation}
The condition $\eta<w$ of Lemma~\ref{lem:signed-count} therefore holds
with $w=(2K)^4$ for the weight $x+2Ky$, and with $w=(2K)^3$ for
either weight $x$ or $y$.
\end{lemma}

\begin{proof}
For a negative residual, write $m_{x,y}=1+dy^2-x^2$.
Since $0<m_{x,y}<K^4$, its valuation is less than $4\log_2K$.
Let $L$ be the greatest valuation attained, or zero if all valuations
vanish. If $N_j$ counts positive-$m_{x,y}$ cells divisible by $2^j$,
then the sum in \eqref{eq:valuation} is $\sum_{j=1}^L N_j$.

For $j\ge3$, a fixed odd residue has at most four square roots modulo
$2^j$. Indeed, for two odd roots $u,v$,
$2^j\mid(u-v)(u+v)$, and one of $u-v,u+v$ has valuation one.
Thus $u\equiv v$ or $-v\pmod{2^{j-1}}$, giving at most four classes
modulo $2^j$.

For $j\ge4$, first count cells with odd $x$. For each fixed $y$ there
are at most four root classes for $x$, each containing at most
$K/2^j+1$ representatives in the square. If $x$ is even and
$m_{x,y}$ is even, then $d,y$ are odd. For each fixed even $x$,
invert $d$ modulo $2^j$ and apply the same bound to
$y^2\equiv d^{-1}(x^2-1)$. These counts give
\[
  N_j\le\frac{8K^2}{2^j}+8K\qquad(j\ge4).
\]
The root count includes all cells with $2^j\mid m_{x,y}$, so it bounds
$N_j$, which counts only those with $m_{x,y}>0$.
For $j=1,2,3$, use $N_j\le K^2$.
Summing gives
\[
  \sum_{j=1}^L N_j\le4K^2+8KL
  <4K^2+32K\log_2K\le7K^2,
\]
because $\log_2K/K\le6/64$ for $K\ge64$.
Now $x+2Ky<2K^2$ on the square, so the combined weighted error is
less than $14K^4<(2K)^4$. Each single coordinate weight is less than
$K$, giving an error less than $7K^3<(2K)^3$.
Both widths satisfy $2^w>K^4$.
\end{proof}

For Pell solutions, $K>X_1$ implies $d\le X_1^2-1<K^2$, as required.
A second width is useful when $K<64$: $w=K^6$ works for the combined
weight whenever $K>X_1$. For $K\ge4$, a nonzero residual has
$\nu_2(|F|)<4\log_2K\le2K$. For $K=3$, the bound follows from $|F|<81$.
Writing
\begin{equation}\label{eq:sigma}
  \sigma=\sum_{x,y<K}x=\sum_{x,y<K}y=\frac{K^2(K-1)}2,
\end{equation}
we obtain
\[
  \eta\le2K(2K+1)\sigma
  =(2K+1)K^3(K-1)<2K^5<K^6.
\]
This width gives larger packed integers than the fourth-degree width.

\subsection{Forming repeated digits by division}

Repeating a digit $x+Cy$ times appears to require a variable-length sum.
Division by $q^{K^2}-1$ produces the repetitions at once, because
$(q^{rK^2}-1)/(q^{K^2}-1)$ is a geometric sum of $r$ powers.
We place the original digits so that the quotient has disjoint copies
and the remainder is too small to alter the integer part.

\begin{lemma}\label{lem:quotient}
Let $q\ge2$, $C\ge0$, and let $D_{x,y}$ be base-$q$ digits on
$\mathcal B_K$, not all equal to $q-1$. Define
\[
  \mathcal P(Q_x,Q_y)=\sum_{x,y<K}Q_x^xQ_y^yD_{x,y}.
\]
Then
\begin{equation}\label{eq:quotient}
  \fl{\frac{\mathcal P(q^{K+K^2},q^{1+CK^2})}{q^{K^2}-1}}
  =\sum_{x,y<K}\sum_{i<x+Cy}q^{xK+y+iK^2}D_{x,y}.
\end{equation}
All positions on the right are distinct and less than $(C+1)K^3$.
The quotient is less than $q^{(C+1)K^3}$.
\end{lemma}

\begin{proof}
Use
\[
  q^{rK^2}=1+(q^{K^2}-1)\sum_{i<r}q^{iK^2}
\]
with $r=x+Cy$ in each term of the numerator.
The remainder is $\sum_{x,y<K}q^{xK+y}D_{x,y}$.
Since $xK+y$ runs through $0,\ldots,K^2-1$, this remainder is
less than $q^{K^2}-1$ under the stated digit hypothesis.
For a position $p=xK+y+iK^2$, division by $K^2$ recovers $i$,
then the remainder recovers $(x,y)$. The position bound follows from
$p<(x+Cy)K^2\le(C+1)(K-1)K^2<(C+1)K^3$.
\end{proof}

Squaring $F$ preserves its zeros and makes every residual nonnegative,
so Lemma~\ref{lem:digits} gives an error-free count with two groups of
bits per digit. The cost is that the packing now involves fourth powers
of the coordinates. Keeping $F$ itself uses only second powers but
requires three groups of bits and a width large enough to absorb the
valuation error.

For signed digits, use $D_{x,y}=\delta_w(F(x,y))$, $q=2^{3w}$.
Each digit is less than $q-1$.
For squared residuals, use $D_{x,y}=\delta_w(F(x,y)^2)$,
$q=2^{2w}$, and $2^w\ge d^2K^4$.
Indeed $|F|\le dK^2$, and the cell $(0,0)$ has $F^2=1$, so its
digit is less than $q-1$.
Taking $C=2K$ in \eqref{eq:quotient} yields
\begin{equation}\label{eq:combined-count}
  \HW(M)=w\bigl((2K+1)\sigma+A+2KB\bigr)+\eta.
\end{equation}
Taking $C=0$ yields
\begin{equation}\label{eq:one-count}
  \HW(M_x)=w(\sigma+A)+\eta_x.
\end{equation}
Transposing the coordinate strides gives
$\HW(M_y)=w(\sigma+B)+\eta_y$.
All three errors vanish for squared residuals. For signed residuals,
Lemma~\ref{lem:valuation} bounds them by the respective widths.

\section{Arithmetic subroutines}\label{sec:arithmetic}

The sums in the counting identities specify integers mathematically.
To express them by a fixed number of arithmetic operations, we evaluate
polynomially weighted geometric progressions. We then give the arithmetic
expression for Hamming weight and the two packing formulas.
Each named subroutine has local variables.

\subsection{Generalized geometric progressions}

The cell residual separates into powers of $x$ and $y$. Consequently its
packing is a combination of products of one-variable sums. Closed formulas
for these sums make the number of operations independent of the number
of cells. For $i\in\{0,2,4\}$, put
\[
  G_i(Q,t)=\sum_{j=0}^t j^iQ^j,
\]
with $0^0=1$. All calls below have $t=K-1$ and $K\ge3$.

\begin{lemma}\label{lem:moments}
Let $u=Q-1$, $Z_Q=Q^K$, $v=ut$, and $Q_+=Q+1$.
For $Q\ge4$, put $z=v-2$ and $g_2=vz+Q_+$. Then
\begin{equation}\label{eq:moments-pair}
  G_0(Q,t)=\fl{Z_Q/u},\qquad
  G_2(Q,t)=\fl{Z_Qg_2/u^3}.
\end{equation}
If $Q\ge7$, also put
\[
  h=Q_+(Q+9),\qquad
  g_4=g_2^2+Q\bigl(4v(z-Q_+)+h\bigr).
\]
Then
\begin{equation}\label{eq:moment-four}
  G_4(Q,t)=\fl{Z_Qg_4/u^5}.
\end{equation}
\end{lemma}

\begin{proof}
Multiplying the defining sums by the appropriate powers of $Q-1$ gives
\begin{align*}
  uG_0&=Z_Q-1,\\
  u^3G_2&=Z_Q(v^2-2v+Q+1)-Q(Q+1),\\
  u^5G_4&=Z_Q\bigl(v^4-4v^3+6(Q+1)v^2
            -4(Q^2+4Q+1)v+a_4(Q)\bigr)-Qa_4(Q),
\end{align*}
where $a_4(Q)=(Q+1)(Q^2+10Q+1)$.
These identities also follow by induction on $t$: they hold at $t=0$,
and their increments are $u^{i+1}(t+1)^iQ^{t+1}$.
Expansion identifies the two numerator polynomials with $g_2$ and $g_4$.
The omitted tails are positive and smaller than their denominators.
For the second moment,
\[
  u^3-Q(Q+1)=Q^2(Q-4)+2Q-1>0.
\]
For the fourth, putting $v_0=Q-7\ge0$ gives
\[
  u^5-Qa_4(Q)=v_0^5+29v_0^4+321v_0^3+1624v_0^2+3336v_0+1056>0.
\]
Also $1<u$. Since each $G_i$ is an integer, dropping these tails
does not change the floors.
\end{proof}

The subroutine $\Gpair(Q,K,t)$ computes the pair $G_0,G_2$ as follows:
\begin{align*}
  u&=Q\tsub1,& Z_Q&=Q^K,& v&=ut,& Q_+&=Q+1,\\
  u_3&=u^3,& G_0&=\fl{Z_Q/u},& z&=v\tsub2,\\
  g_2&=vz+Q_+,& G_2&=\fl{Z_Qg_2/u_3}.&&
\end{align*}
\emph{Operation count:} $11$, or $10$ when $Z_Q=Q^K$ is supplied.

The subroutine $\Gtriple(Q,K,t)$ first performs the preceding
assignments and then computes
\begin{align*}
  u_5&=u^5,& h&=Q_+(Q+9),\\
  g_4&=g_2g_2+Q\bigl((4v)(z\tsub Q_+)+h\bigr),&
  G_4&=\fl{Z_Qg_4/u_5}.
\end{align*}
\emph{Operation count:} $11+1+2+7+2=23$, or $22$ when $Z_Q$ is
supplied. The subtraction in the pair is exact for $Q\ge4$.
For the triple, $z-Q_+=v-Q-3\ge Q-5>0$.
The programs compute shared quantities once and use them to return
the moments together.

\subsection{An arithmetic expression for Hamming weight}\label{sec:hw}

The packing method needs a count of binary ones, which is not one of our
primitive operations. We obtain it as the exponent of two dividing a
central binomial coefficient. Integer powers produce that coefficient,
and a gcd with a power of two isolates its valuation. This construction
has a fixed operation count, although its intermediate integers will
dominate the size analysis in Section~\ref{sec:expanded-size}.

Let $m\ge1$. Counting the factors of two in factorials gives
\[
  \nu_2(k!)=\sum_{j\ge1}\fl{k/2^j}=k-\HW(k).
\]
The second equality follows by summing the contribution of each binary
digit of $k$. Since $\HW(2m)=\HW(m)$, it follows that
\begin{equation}\label{eq:kummer}
  \HW(m)=\nu_2\binom{2m}{m}.
\end{equation}
This is the central-binomial case of Kummer's theorem used in
\cite{PSh24}. A base-$2^{2m}$ digit extraction evaluates the binomial
coefficient: the coefficients of $(1+z)^{2m}$ are less than $2^{2m}$,
so
\begin{equation}\label{eq:binomial}
  \binom{2m}{m}
  =\fl{\frac{(2^{2m}+1)^{2m}}{(2^{2m})^m}}\md 2^{2m}.
\end{equation}

We will use a positive integer $e$ satisfying
\begin{equation}\label{eq:e-condition}
  \HW(m)<e\le2m.
\end{equation}
For $\Pi=2^e$ and
$\alpha=\binom{2m}{m}\md\Pi$, the valuation condition implies
$\alpha>0$ and
\[
  \gcd(\alpha,\Pi)=2^{\HW(m)}.
\]
The following identity supplies the gcd as an arithmetic term.

\begin{lemma}\label{lem:gcd}
For positive integers $\alpha,\beta$ with $\alpha\beta\ge2$, put
$\gamma=\alpha\beta$. Then
\begin{equation}\label{eq:gcd}
  \gcd(\alpha,\beta)=
  \left(\fl{\frac{2^{\gamma(\gamma+\alpha+\beta)}}
  {(2^{\gamma\alpha}-1)(2^{\gamma\beta}-1)}}
  \md 2^\gamma\right)-1.
\end{equation}
\end{lemma}

\begin{proof} This was originally proved by Prunescu and Shunia in \cite{PSg24}. Here we show a shorter proof. 
Set $z=2^\gamma$, and let $a_r$ count the nonnegative integer
solutions of $\alpha i+\beta j=r$. Geometric expansion gives
\[
  \frac{z^{\gamma+\alpha+\beta}}
       {(z^\alpha-1)(z^\beta-1)}
  =\sum_{i,j\ge0}z^{\gamma-\alpha i-\beta j}
  =\sum_{r=0}^{\gamma}a_rz^{\gamma-r}
   +\sum_{k\ge1}a_{\gamma+k}z^{-k}.
\]
Since $\max(\alpha,\beta)\ge2$, one has $a_r\le r/2+1$.
The expression $(\gamma/2+1)/(2^\gamma-1)
+2^\gamma/(2(2^\gamma-1)^2)$ equals $8/9$ at $\gamma=2$,
and both summands decrease for integer $\gamma\ge2$.
The final sum is therefore at most
\[
  \frac{\gamma/2+1}{z-1}
   +\frac{z}{2(z-1)^2}\le\frac89<1.
\]
Thus the floor discards the final sum. Reduction modulo $z$
leaves $a_\gamma=\gcd(\alpha,\beta)+1$, which is less than $z$.
Indeed, if $g=\gcd(\alpha,\beta)$, the solutions at $r=\alpha\beta$
are $i=(\beta/g)k$, $j=(\alpha/g)(g-k)$ for $0\le k\le g$.
\end{proof}

To read a valuation from $g=2^\nu$ with $\nu<e$, put $u=2^e-1$.
Then
\[
  g^e=(u+1)^\nu\equiv1+\nu u\pmod{u^2}.
\]
Here $e\ge2$, and $1+\nu u<u^2$. Dividing the remainder by
$u$ and taking the floor gives $\nu$.
This valuation extraction is also used in \cite{Mazzanti,Marchenkov}.

The resulting subroutine $\mathsf H(m,e)$ consists of three parts.

\textbf{H1. Binomial residue.} Compute
\[
  a=2m,\qquad L=2^{a},\qquad
  \alpha=\fl{(L+1)^{a}/L^m}\md\Pi.
\]
For the default choice, set $e=a$ and $\Pi=L$ by reuse.
For a supplied $e$, compute $\Pi=2^e$ before the last assignment.
Because $\Pi\mid L$, the output is the desired binomial residue.

\emph{Operation count:} $7$ for the default choice, or $8$ with a
supplied $e$.

\textbf{H2. Gcd.} Compute
\[
  \gamma=\alpha\Pi,\qquad
  g=\left(\fl{\frac{2^{\gamma(\gamma+\alpha+\Pi)}}
  {(2^{\gamma\alpha}\tsub1)(2^{\gamma\Pi}\tsub1)}}
  \md 2^\gamma\right)\tsub1.
\]
\emph{Operation count:} $16$.

\textbf{H3. Valuation.} Compute
\[
  u=\Pi\tsub1,\qquad
  \mathsf H(m,e)=\fl{(g^e\md u^2)/u}.
\]
\emph{Operation count:} $5$.

\emph{Total operation count:} $28$ with $e=2m$, or $29$ with a
supplied $e$ satisfying \eqref{eq:e-condition}.
The default is admissible because $\HW(m)\le m<2m$.
We write $\mathsf H(m)$ for this default $28$-operation program.

\subsection{Evaluating the cell packings}\label{sec:packing-formulas}

For a width $w$, let $P=2^w$ and $P_-=P\tsub1$.
At bases $Q_x,Q_y$, write
\[
  U_i=G_i(Q_x,K-1),\qquad V_i=G_i(Q_y,K-1).
\]
The signed packing is
\[
  \mathcal P_s(Q_x,Q_y)
  =\sum_{x,y<K}Q_x^xQ_y^y\delta_w(F(x,y)).
\]
Expanding $P+1-F=P+2-x^2+dy^2$ yields the subroutine
\begin{equation}\label{eq:signed-packing}
  \Signed(U,V)
  =P_-\bigl[V_0((P+2)U_0\tsub U_2)+U_0(dV_2)\bigr].
\end{equation}
\emph{Operation count:} $8$, with $d,P,P_-$ supplied.
The subtraction is exact when $P>(K-1)^2$.

The squared packing is
\[
  \mathcal P_q(Q_x,Q_y)
  =\sum_{x,y<K}Q_x^xQ_y^y\delta_w(F(x,y)^2).
\]
For a supplied $d_2=d^2$, its subroutine is
\begin{align}
  D_2&=dV_2,\notag\\
  \Squared(U,V)
    &=P_-\bigl[V_0(PU_0\tsub U_4)
       +2U_2(V_0+D_2)\tsub U_0(2D_2+d_2V_4)\bigr].
       \label{eq:squared-packing}
\end{align}
\emph{Operation count:} $14$, with $d,d_2,P,P_-$ supplied.
Computing $d_2$ costs one additional operation, shared if two packings
are formed.

For verification, expand
\[
  P+1-F^2=(P-x^4)+2x^2(1+dy^2)-(2dy^2+d^2y^4).
\]
Summing its three terms gives \eqref{eq:squared-packing}.
Under $P\ge d^2K^4$, the first subtraction is exact term by term.
The final subtraction is exact because the resulting polynomial is
$P+1-F^2\ge1$ throughout the square.
\section{Five programs for the fundamental solution}\label{sec:programs}

Each program in this section takes $d,K,w$ as inputs. The conditions on
$K,w$ are stated separately for each construction.
Calls to $\Gpair$, $\Gtriple$, $\Signed$, $\Squared$, $\Coord$, and
$\Recover$ mean substitution of the straight-line subroutines already
displayed. The Hamming-weight call is also expanded in all total counts.
The costs of constructing $K,w$ from $d$ are added in
Section~\ref{sec:parameters}. Programs SC and SO display the common
base and coordinate-extraction subroutines in full. The other programs
call these named subroutines and give their complete sequences separately.

\subsection{Signed residual with the combined weight}\label{sec:sc}

The weight $x+2Ky$ produces $A+2KB$. It requires one Hamming weight
and no binomial approximation for coordinate recovery.

\begin{theorem}\label{thm:sc}
Let $K>X_1$, $K\ge64$, and $w=(2K)^4$.
Program SC below returns $(X_1,Y_1)$.
Its packed integer has at most $3w(2K+1)K^3=O(K^8)$ binary digits,
and its intermediates outside the Hamming-weight subroutine have
$O(K^8)$ binary digits.
\end{theorem}

\textbf{SC1. Combined-weight bases $\mathsf B_C^{(3)}(K,w)$.}
\begin{align*}
  t&=K\tsub1,& P&=2^w,& P_-&=P\tsub1,& q&=P^3,\\
  q_1&=q^K,& q_2&=q_1^K,& q_3&=q_2^K,\\
  Q_x&=q_1q_2,& Q_y&=q(q_3q_3).&&
\end{align*}
\emph{Operation count:} $10$.
The subroutine returns all nine displayed values. We also define
$\mathsf B_C^{(2)}(K,w)$ by using $q=PP$ in place of $q=P^3$,
with the same count.

\textbf{SC2. Moments.}
\[
  (U_0,U_2)=\Gpair(Q_x,K,t),\qquad
  (V_0,V_2)=\Gpair(Q_y,K,t).
\]
\emph{Operation count:} $11+11=22$.

\textbf{SC3. Cell packing.}
\[
  T_s=\Signed(U,V).
\]
\emph{Operation count:} $8$.

\textbf{SC4. Weighted packing.}
\[
  M=\fl{T_s/(q_2\tsub1)}.
\]
\emph{Operation count:} $2$.

\textbf{SC5. Hamming weight.}
\[
  H_M=\mathsf H(M).
\]
\emph{Operation count:} $28$.

\textbf{SC6. Coordinate sums $\mathsf E_C(K,w,t,H_M)$.}
\begin{align*}
  C&=2K,& \sigma&=\fl{(KK)t/2},\\
  C_+&=C+1,& W_C&=C_+\sigma,\\
  R_C&=\fl{H_M/w}\tsub W_C,\\
  A&=R_C\md C,& B&=\fl{R_C/C}.
\end{align*}
\emph{Operation count:} $10$.
This subroutine returns $(A,B)$.

\textbf{SC7. Fundamental solution.}
\[
  (X_1,Y_1)=\Recover(d,A,B).
\]
\emph{Operation count:} $12$.

\emph{Total operation count:}
\[
  10+22+8+2+28+10+12=92.
\]
There are $64$ operations outside one Hamming-weight call.

\begin{proof}
The bases have exponents $K+K^2$ and $1+2K^3$, so
Lemma~\ref{lem:quotient} applies with $C=2K$.
By \eqref{eq:combined-count} and Lemma~\ref{lem:valuation},
$R_C=A+CB$. Lemma~\ref{lem:sum-bound} gives $A<C$, which proves
the quotient and remainder assignments in SC6.
Theorem~\ref{lem:identity} proves SC7.
The positions in $M$ are less than $(C+1)K^3$, and each position has
$3w$ bits. The progression powers and their polynomial factors have
$O(wK^4)=O(K^8)$ bits as well.
\end{proof}

The same program is valid for any $K>X_1$ and width satisfying
$2^w>\max_{\mathcal B_K}|F|$ and
$\sum_{F<0}(x+2Ky)\nu_2(-F)<w$.
In particular, the width $K^6$ established after
Lemma~\ref{lem:valuation} applies for all $K>X_1$, with
$O(K^{10})$ packed bits.

\subsection{Signed residual with one coordinate sum}\label{sec:so}

Encoding $A$ alone uses the weight $x$. A call to $\Coord$ then
recovers $B$ from $A$, using powers already computed for the packing.

\begin{theorem}\label{thm:so}
Let $K>X_1$, $K\ge64$, and $w=(2K)^3$.
Program SO below returns $(X_1,Y_1)$.
Its packed integer has at most $3wK^3=O(K^6)$ binary digits.
The intermediates outside Hamming weight have
$O(K^6+dK^5)$ binary digits, and hence $O(K^6)$ when $K\ge d$.
\end{theorem}

\textbf{SO1. Single-weight bases $\mathsf B_O^{(3)}(K,w)$.}
\begin{align*}
  t&=K\tsub1,& P&=2^w,& P_-&=P\tsub1,& q&=P^3,\\
  q_1&=q^K,& q_2&=q_1^K,& Q&=q_1q_2.&&
\end{align*}
\emph{Operation count:} $7$.
The subroutine returns all seven displayed values. Its squared-residual counterpart
$\mathsf B_O^{(2)}(K,w)$ uses $q=PP$ and has the same count.

\textbf{SO2. Moments.}
\[
  (U_0,U_2)=\Gpair(Q,K,t),\qquad
  (V_0,V_2)=\Gpair(q,K,t,q_1).
\]
The fourth argument supplies the previously computed power $q^K$.

\emph{Operation count:} $11+10=21$.

\textbf{SO3. Cell packing.}
\[
  T_s=\Signed(U,V).
\]
\emph{Operation count:} $8$.

\textbf{SO4. Weighted packing.}
\[
  M=\fl{T_s/(q_2\tsub1)}.
\]
\emph{Operation count:} $2$.

\textbf{SO5. Hamming weight.}
\[
  H_M=\mathsf H(M).
\]
\emph{Operation count:} $28$.

\textbf{SO6. First coordinate sum $\mathsf E_O(K,w,t,H_M)$.}
\[
  \sigma=\fl{(KK)t/2},\qquad A=\fl{H_M/w}\tsub\sigma.
\]
\emph{Operation count:} $5$.
This subroutine returns $A$.

\textbf{SO7. Second coordinate sum.}
\[
  B=\Coord(d,K,A,q_1,q_2).
\]
\emph{Operation count:} $15$.

\textbf{SO8. Fundamental solution.}
\[
  (X_1,Y_1)=\Recover(d,A,B).
\]
\emph{Operation count:} $12$.

\emph{Total operation count:}
\[
  7+21+8+2+28+5+15+12=98.
\]
There are $70$ operations outside one Hamming-weight call.

\begin{proof}
Lemma~\ref{lem:quotient} with $C=0$ repeats cell $(x,y)$ exactly
$x$ times. Its valuation error is less than $7K^3<w$, so
\eqref{eq:one-count} proves SO6.
Here $q_1=2^{3wK}$, and $3w\ge2d$ since $d<K^2$.
Thus the hypotheses of $\Coord$ hold, proving SO7.
The reconstruction subroutine then gives the fundamental solution.
The packed length follows from the position bound $K^3$.
The largest additional recovery power, $(q_1+1)^{2dK}$, has
$6wdK^2+1$ bits. All other intermediates outside Hamming weight
have $O(wK^3)$ bits.
\end{proof}

For a smaller box, the same program applies whenever
$2^w>\max|F|$, $\sum_{F<0}x\nu_2(-F)<w$, and $3w\ge2d$.

\subsection{Squared residual with the combined weight}\label{sec:qc}

Squaring $F$ removes the valuation error. The resulting packing uses
fourth moments and can take a width logarithmic in $dK$.

\begin{theorem}\label{thm:qc}
Let $K>X_1$ and $2^w\ge d^2K^4$.
Program QC below returns $(X_1,Y_1)$.
Its packed integer has at most $2w(2K+1)K^3$ binary digits.
All intermediates outside Hamming weight have $O(wK^4)$ binary digits.
\end{theorem}

\textbf{QC1. Bases.}
\[
  (t,P,P_-,q,q_1,q_2,q_3,Q_x,Q_y)=\mathsf B_C^{(2)}(K,w).
\]
\emph{Operation count:} $10$.

\textbf{QC2. Moments and cell packing.}
\begin{gather*}
  (U_0,U_2,U_4)=\Gtriple(Q_x,K,t),\qquad
  (V_0,V_2,V_4)=\Gtriple(Q_y,K,t),\\
  d_2=dd,\qquad T_q=\Squared(U,V).
\end{gather*}
\emph{Operation count:} $23+23+1+14=61$.

\textbf{QC3. Weighted counting and coordinate sums.}
\begin{gather*}
  M=\fl{T_q/(q_2\tsub1)},\qquad H_M=\mathsf H(M),\\
  (A,B)=\mathsf E_C(K,w,t,H_M).
\end{gather*}
\emph{Operation count:} $2+28+10=40$.

\textbf{QC4. Fundamental solution.}
\[
  (X_1,Y_1)=\Recover(d,A,B).
\]
\emph{Operation count:} $12$.

\emph{Total operation count:} $10+61+40+12=123$.
There are $95$ operations outside one Hamming-weight call.

\begin{proof}
The bound on $w$ makes $F^2\le2^w$ throughout the square.
Equation~\eqref{eq:combined-count} now has $\eta=0$.
Division and remainder recover $A,B$ because $A<2K$, and
Theorem~\ref{lem:identity} gives the output.
There are fewer than $(2K+1)K^3$ possible occupied positions,
each of width $2w$. The moment numerators have the same
$O(wK^4)$ bit-length order.
\end{proof}

\subsection{Squared residual with one coordinate sum}\label{sec:qo}

The squared residual can also be combined with recovery from $A$.
The condition $w\ge d$ ensures that the packing powers are large
enough for the binomial extractions in $\Coord$.

\begin{theorem}\label{thm:qo}
Let $K>X_1$, $w\ge d$, and $2^w\ge d^2K^4$.
Program QO below returns $(X_1,Y_1)$.
Its packed integer has at most $2wK^3$ binary digits.
The intermediates outside Hamming weight have
$O(wK^3+wdK^2)$ binary digits, and hence $O(wK^3)$ when $K\ge d$.
\end{theorem}

\textbf{QO1. Bases.}
\[
  (t,P,P_-,q,q_1,q_2,Q)=\mathsf B_O^{(2)}(K,w).
\]
\emph{Operation count:} $7$.

\textbf{QO2. Moments and cell packing.}
\begin{gather*}
  (U_0,U_2,U_4)=\Gtriple(Q,K,t),\qquad
  (V_0,V_2,V_4)=\Gtriple(q,K,t,q_1),\\
  d_2=dd,\qquad T_q=\Squared(U,V).
\end{gather*}
\emph{Operation count:} $23+22+1+14=60$.

\textbf{QO3. Weighted counting and first coordinate sum.}
\begin{gather*}
  M=\fl{T_q/(q_2\tsub1)},\qquad H_M=\mathsf H(M),\\
  A=\mathsf E_O(K,w,t,H_M).
\end{gather*}
\emph{Operation count:} $2+28+5=35$.

\textbf{QO4. Coordinate recovery and fundamental solution.}
\[
  B=\Coord(d,K,A,q_1,q_2),\qquad (X_1,Y_1)=\Recover(d,A,B).
\]
\emph{Operation count:} $15+12=27$.

\emph{Total operation count:} $7+60+35+27=129$.
There are $101$ operations outside one Hamming-weight call.

\begin{proof}
Equation~\eqref{eq:one-count} has $\eta_x=0$, proving QO3.
For coordinate recovery, $q_1=2^{2wK}$ and $2w\ge2d$.
Theorems~\ref{thm:coordinate} and~\ref{lem:identity} therefore
prove the output assignments. The quotient occupies positions below
$K^3$ in base $2^{2w}$.
The recovery power $(q_1+1)^{2dK}$ has $4wdK^2+1$ bits, giving
the additional term in the intermediate-length bound.
\end{proof}

\subsection{Squared residual with two separate coordinate sums}\label{sec:qt}

The final construction extracts $A$ and $B$ separately. It uses two
Hamming weights and does not require the binomial approximation to
$\sqrt d$ or the additional condition $w\ge d$.

\begin{theorem}\label{thm:qt}
Let $K>X_1$ and $2^w\ge d^2K^4$.
Program QT below returns $(X_1,Y_1)$.
Each of its two packed integers has at most $2wK^3$ binary digits.
All intermediates outside Hamming weight have $O(wK^3)$ binary digits.
\end{theorem}

\textbf{QT1. Bases.}
\[
  (t,P,P_-,q,q_1,q_2,Q)=\mathsf B_O^{(2)}(K,w).
\]
\emph{Operation count:} $7$.

\textbf{QT2. Shared moments and the two cell packings.}
\begin{gather*}
  (U_0,U_2,U_4)=\Gtriple(Q,K,t),\qquad
  (V_0,V_2,V_4)=\Gtriple(q,K,t,q_1),\\
  d_2=dd,\qquad T_x=\Squared(U,V),\qquad T_y=\Squared(V,U).
\end{gather*}
Both packing calls share $d_2$, $P$, and $P_-$.

\emph{Operation count:} $23+22+1+14+14=74$.

\textbf{QT3. Weighted counting and coordinate sums.}
\begin{gather*}
  D_q=q_2\tsub1,\qquad M_x=\fl{T_x/D_q},\qquad M_y=\fl{T_y/D_q},\\
  H_x=\mathsf H(M_x),\qquad H_y=\mathsf H(M_y),\\
  \sigma=\fl{(KK)t/2},\qquad
  A=\fl{H_x/w}\tsub\sigma,\qquad B=\fl{H_y/w}\tsub\sigma.
\end{gather*}
\emph{Operation count:} $3+28+28+7=66$.

\textbf{QT4. Fundamental solution.}
\[
  (X_1,Y_1)=\Recover(d,A,B).
\]
\emph{Operation count:} $12$.

\emph{Total operation count:} $7+74+66+12=159$.
There are $103$ operations outside two Hamming-weight calls.

\begin{proof}
The first quotient is the case $C=0$ of Lemma~\ref{lem:quotient}.
For the second, use the position $yK+x$ instead of $xK+y$.
The square is unchanged, and the quotient repeats cell $(x,y)$
exactly $y$ times. The squared digits have no valuation error,
so the two Hamming weights give $\sigma+A$ and $\sigma+B$ after
division by $w$. Theorem~\ref{lem:identity} finishes the proof.
Each quotient occupies positions below $K^3$. The moment triples
are shared between the two transposed packings, and the triple at
$q$ reuses $q_1=q^K$.
\end{proof}
\section{Explicit parameters and integer sizes}\label{sec:parameters}

Table~\ref{tab:cores} compares the five programs for supplied $K,w$.
The signed programs use polynomial widths in $K$ to absorb the valuation
error. The squared programs can use a width of order $\log(dK)$,
subject to the additional condition $w\ge d$ in Program QO.
These bounds compare the inputs to Hamming weight. They do not yet
measure the largest integers formed by the complete programs. With either
parameter family below, all five programs have the same exponential
height in their full size bounds. Choosing a smaller parameter within
Hamming weight changes that height, as Table~\ref{tab:parameters} shows.

\begin{table}[htbp]
\centering
\small
\begin{tabular}{@{}lrrl@{}}
\toprule
Program & Operations & HW calls & Bits per packed integer\\
\midrule
SC: signed, combined weight &92&1&$O(K^8)$\\
SO: signed, one sum &98&1&$O(K^6)$\\
QC: squared, combined weight &123&1&$O(wK^4)$\\
QO: squared, one sum &129&1&$O(wK^3)$\\
QT: squared, two sums &159&2&$O(wK^3)$\\
\bottomrule
\end{tabular}
\caption{Counts for the five fundamental-solution programs, excluding
construction of $K,w$ and expanding each HW call by its $28$-operation
program. The width and domain are those of the corresponding theorem.
Bit lengths refer to arguments of HW. Program QT has two such integers.
The other programs have one.}
\label{tab:cores}
\end{table}

\subsection{An elementary bound and parameter programs}

We first give a bound that provides parameters by two additional
operations for any of the five programs.

\begin{proposition}\label{prop:elementary}
For nonsquare $d\ge2$, one has $\eps<\mathrm e^{4d}<64^d$.
\end{proposition}

\begin{proof}
Put $\theta=\sqrt d$ and $m_0=\fl\theta$.
Let $\zeta_1,\ldots,\zeta_\ell$ be the complete quotients in the
least period of the continued fraction of $\sqrt d$.
The complete quotients have the form
$\zeta_i=(\theta+a_i)/b_i$. In the periodic part they satisfy
$\zeta_i>1$ and $-1<\overline{\zeta_i}<0$.
The continued-fraction identities give
\[
  1\le a_i\le m_0,\qquad 1\le b_i<2\theta,
\]
and the pairs $(a_i,b_i)$ in a least period are distinct.
These standard identities and the Pell identities for convergents are
recorded, for example, in \cite{HardyWright}.
For the convergent $p_{\ell-1}/q_{\ell-1}$ they give
\[
  p_{\ell-1}-\theta q_{\ell-1}
    =\frac{(-1)^\ell}{\zeta_1\cdots\zeta_\ell},\qquad
  p_{\ell-1}^2-dq_{\ell-1}^2=(-1)^\ell.
\]
Consequently
$u_0=\prod_{i=1}^{\ell}\zeta_i=p_{\ell-1}+q_{\ell-1}\sqrt d$.
Its square is a nontrivial positive norm-one unit in $\Z[\sqrt d]$,
so $\eps\le u_0^2$. This argument does not require $d$ to be squarefree.

Set $L_0=\fl{2\theta}$. Each denominator $b_i$ occurs at most
$m_0$ times, and $1<\zeta_i<2\theta/b_i$.
All factors $2\theta/Q$ for $1\le Q\le L_0$ exceed one.
Adding missing factors gives
\[
  u_0<\prod_{Q=1}^{L_0}\left(\frac{2\theta}{Q}\right)^{m_0}
      =\left(\frac{(2\theta)^{L_0}}{L_0!}\right)^{m_0}
      <\mathrm e^{2\theta m_0}<\mathrm e^{2d}.
\]
The penultimate inequality bounds one term of the exponential series
by its full sum. Squaring proves the first bound, and
$\mathrm e^4<64$ proves the second.
\end{proof}

The elementary parameter programs are as follows. Assignments taken
from a main program remain counted there, even if moved earlier to
form its width.

\textbf{E-SC. Parameters for Program SC.}
\[
  K=64^d,\qquad w=C^4,
\]
where the assignment $C=2K$ from SC6 is performed before computing $w$.

\emph{Additional operation count:} $2$.
The complete program in $d$ has $92+2=94$ operations.

\textbf{E-SO. Parameters for Program SO.}
\[
  K=64^d,\qquad w=h^3,
\]
where $h=2K$ is the first assignment of $\Coord$ in SO7 and is
performed before computing $w$.

\emph{Additional operation count:} $2$.
The complete program in $d$ has $98+2=100$ operations.

\textbf{E-Q. Parameters for each of Programs QC, QO, QT.}
\[
  K=64^d,\qquad w=26d.
\]
\emph{Additional operation count:} $2$.
The complete programs have respectively $125$, $131$, and $161$ operations.

Proposition~\ref{prop:elementary} gives $K>X_1$, and these choices
satisfy $K\ge4096$ and $K\ge d$.
For the squared programs,
\[
  d^2K^4=d^2 2^{24d}\le2^{26d}=2^w,
\]
because $d\le2^d$. The condition $w\ge d$ also holds.
All five programs with these parameters have $2^{O(d)}$ packed bits.
The constants in the exponent depend on the chosen construction.

\subsection{Parameters from Hua's bound}

Hua's bound \cite{Hua42}, in the form for Pell's equation stated in
\cite[p.~185]{Len02}, is
\begin{equation}\label{eq:hua}
  \log\eps<\sqrt d\bigl(\log(4d)+2\bigr),\qquad
  X_1<(4\mathrm e^2d)^{\sqrt d}.
\end{equation}
Here $\eps$ is the least unit greater than one of norm $+1$ in
$\Z[\sqrt d]$. In the discriminant formulation
$u^2-\Delta v^2=4$, the choice $\Delta=4d$ gives $u=2X_1$, $v=Y_1$,
and $(u+v\sqrt\Delta)/2=\eps$.
The regulator in the quoted bound belongs to the norm-one subgroup.
If the least unit greater than one has norm $-1$, its square is $\eps$,
but \eqref{eq:hua} already bounds $\log\eps$, so its constant is not doubled.
To turn the exponent bound into a term, we use the same central-binomial
ratio as in coordinate recovery.

\begin{lemma}\label{lem:root-parameter}
For $d\ge2$, let $r=\fl{R_d}$, where
$R_d=4^d/\binom{2d}{d}$. Then $\sqrt d<r\le2\sqrt d$.
\end{lemma}

\begin{proof}
The recurrence $R_{j+1}/R_j=2(j+1)/(2j+1)$ and induction give
\[
  \sqrt{3j+1}\le R_j\le2\sqrt j\qquad(j\ge1).
\]
For the lower bound the inductive comparison is
\[
  4(j+1)^2(3j+1)-(2j+1)^2(3j+4)=j\ge0.
\]
For the upper bound it is $4j(j+1)<(2j+1)^2$.
Since $\sqrt{3d+1}>\sqrt d+1$ for $d>1$, taking the floor proves
$r>\sqrt d$. The upper bound follows directly.
\end{proof}

\textbf{H-P1. Rational upper bound for the exponent.}
\[
  h_d=2d,\qquad L_d=2^{h_d},\qquad
  c_d=\fl{(L_d+1)^{h_d}/L_d^d}\md L_d,
  \qquad r=\fl{L_d/c_d}.
\]
\emph{Operation count:} $1+1+5+1=8$.
Digit extraction gives $c_d=\binom{2d}{d}$, so the last output
is the $r$ of Lemma~\ref{lem:root-parameter}.

\textbf{H-P2. Square size.}
\[
  K=(32d)^r.
\]
\emph{Operation count:} $2$.
Because $32>4\mathrm e^2$ and $r>\sqrt d$,
\eqref{eq:hua} gives $K>X_1$.
Also $K\ge4096$, $K\ge d$, and $\log_2K=O(\sqrt d\log d)$.

\textbf{H-SC. Width for Program SC.}
\[
  w=C^4,
\]
reusing the assignment $C=2K$ from SC6 before constructing the width.

\emph{Additional width count:} $1$.
Including H-P1 and H-P2, the complete program has
$92+8+2+1=103$ operations.

\textbf{H-SO. Width for Program SO.}
\[
  w=h^3,
\]
reusing the assignment $h=2K$ from SO7 before constructing the width.

\emph{Additional width count:} $1$.
The complete program has $98+8+2+1=109$ operations.

\textbf{H-Q. Width for each of Programs QC, QO, QT.}
\[
  w=(4r+2)(d+5).
\]
\emph{Operation count:} $4$.
Together with H-P1 and H-P2, these parameters cost $14$ operations.
The complete programs have respectively $137$, $143$, and $173$ operations.

Indeed,
\[
  d^2K^4\le(32d)^{4r+2}
  \le2^{(4r+2)(d+5)}=2^w,
\]
using $\log_2(32d)\le d+5$. Also $w\ge d$.
All five choices have $2^{O(\sqrt d\log d)}$ packed bits.
No further Hamming weight is used to construct these parameters.
The complete counts are collected in Table~\ref{tab:parameters}.

\begin{table}[htbp]
\centering
\small
\begin{tabular}{@{}lrrrr@{}}
\toprule
&\multicolumn{2}{c}{Elementary parameters}&\multicolumn{2}{c}{Hua parameters}\\
\cmidrule(lr){2-3}\cmidrule(l){4-5}
Program & Default HW & Smaller $e$ & Default HW & Smaller $e$\\
\midrule
SC &94&98&103&107\\
SO &100&103&109&112\\
QC &125&129&137&141\\
QO &131&134&143&146\\
QT &161&166&173&178\\
\bottomrule
\end{tabular}

\medskip
Largest intermediate bit-length bounds for all five programs:
\medskip

{\renewcommand{\arraystretch}{1.6}
\begin{tabular}{@{}lll@{}}
\toprule
Parameters & Default HW & Smaller $e$\\
\midrule
Elementary & $2^{2^{2^{O(d)}}}$ & $2^{2^{O(d)}}$\\
Hua & $2^{2^{2^{O(\sqrt d\log d)}}}$ & $2^{2^{O(\sqrt d\log d)}}$\\
\bottomrule
\end{tabular}}
\caption{Operation counts and full intermediate-size bounds for the
fundamental-solution programs on nonsquare $d\ge2$. Counts include
parameter construction and the arithmetic expansion of HW. The default
uses $e=2M$ for each HW input $M$. The smaller choices of $e$ are HW-C,
HW-O, and HW-T in Section~\ref{sec:expanded-size}. The size bounds apply
to the entire program, with constants depending on the construction.}
\label{tab:parameters}
\end{table}

\subsection{Lengths after expanding Hamming weight}\label{sec:expanded-size}

The preceding bit bounds concern the arguments of the Hamming Weight function $\mathsf H$. The arithmetic expansion of $\mathsf H(M,e)$ forms much larger integers.
Its binomial-extraction power $(2^{2M}+1)^{2M}$ has $4M^2+1$ bits.
The gcd numerator is a power of two with exponent
\[
  \alpha\Pi(\alpha\Pi+\alpha+\Pi)<2^{4e+1}.
\]
When $e=2M$, we have $\Pi=4^M$ and $\alpha=\binom{2M}{M}$.
The gcd numerator exponent is then asymptotic to
$2^{8M}/(\pi M)$, by the central-binomial asymptotic.
This integer is larger than the binomial-extraction power. If $M$ has
at most $\lambda$ bits, the default HW expansion has at most
$2^{2^{O(\lambda)}}$ bits in any intermediate. Substituting the packed
lengths gives the triply exponential bounds
\[
  2^{2^{2^{O(d)}}}\quad\hbox{and}\quad
  2^{2^{2^{O(\sqrt d\log d)}}}
\]
for elementary and Hua parameters, respectively.

The default $e=2M$ uses the value of the packed integer as a bound for
its Hamming weight. The digit construction gives a much smaller bound
directly from the number of repeated cells. Using that bound in the gcd
step makes the binomial-extraction power the largest intermediate in HW
and removes one exponential from the full size bound.
Smaller admissible values of $e$ follow from the known number of
occupied digit positions. They give the following substitutions in the
five straight-line programs. Before each HW call, move the assignments
for $\sigma$, $2K$, and, when used, $C_+$ ahead of that call.
These quantities depend on $K$ and remain counted once in the main
program.

\textbf{HW-C. For SC and QC.} Reuse $C=2K$, $C_+=C+1$, and
$\sigma$, and compute
\[
  e=wC_+(\sigma+C).
\]
Replace the corresponding call by $\mathsf H(M,e)$.

\emph{Additional operation count:} $3$ for $e$ and $1$ inside HW,
giving totals $96$ for SC and $127$ for QC, before parameters.
For SC the totals including parameters are $98$ and $107$.

\textbf{HW-O. For SO and QO.} Reuse $h=2K$ from $\Coord$ and
compute
\[
  e=w(\sigma+h).
\]
Replace the corresponding call by $\mathsf H(M,e)$.

\emph{Additional operation count:} $2$ for $e$ and $1$ inside HW,
giving totals $101$ for SO and $132$ for QO, before parameters.

\textbf{HW-T. For QT.} Compute
\[
  h_T=2K,\qquad e=w(\sigma+h_T),
\]
and replace the two calls by $\mathsf H(M_x,e)$ and $\mathsf H(M_y,e)$.

\emph{Additional operation count:} $3$ for $e$ and $1$ in each
separate HW expansion, giving $164$ operations before parameters.

\begin{proposition}\label{prop:expanded-size}
The three choices above satisfy \eqref{eq:e-condition} for their
respective arguments. With these choices and $K\ge d$, the largest
intermediate bit lengths are bounded by
\[
  2^{O(K^8)}\ \text{for SC},\qquad
  2^{O(K^6)}\ \text{for SO},\qquad
  2^{O(wK^4)}\ \text{for QC},\qquad
  2^{O(wK^3)}\ \text{for QO and QT}.
\]
Under the elementary parameters these are $2^{2^{O(d)}}$, and
under the Hua parameters they are $2^{2^{O(\sqrt d\log d)}}$.
\end{proposition}

\begin{proof}
For the combined weight, $A,B<C$, so the integer
$A+CB\le C(C+1)-1$. Equation~\eqref{eq:combined-count}, with
$0\le\eta<w$, therefore gives $\HW(M)<w(C+1)(\sigma+C)$.
For a single coordinate weight, $A,B<2K$ gives
$\HW(M)<w(\sigma+2K)$, including when its error is nonzero.

The last copy of cell $(K-1,K-1)$ occurs at position
\[
  p_C=(2K+1)(K-1)K^2-1
\]
for the combined weight, and at
$p_O=(K-1)K^2-1$ for either single weight.
Its digit is positive, so $M\ge q^p$ for the relevant position $p$.
For $K\ge3$,
\[
  (2K+1)(\sigma+2K)<p_C,\qquad \sigma+2K<p_O.
\]
For the first inequality the difference is
$(2K+1)(K^2(K-1)/2-2K)-1>0$.
For the second it is $K^2(K-1)/2-2K-1>0$.
Thus $e<wp$, and $e\le2M$ follows from $M\ge2^{2wp}$.
Also
\[
  2^{4e+1}<2^{4wp+2}\le4M^2.
\]
The gcd numerator has at most $2^{4e+1}$ bits, so within each HW
call it is smaller than the binomial-extraction power.
If the packed argument has at most $\lambda$ bits, the expanded
length is consequently $2^{O(\lambda)}$.
The bounds from the five theorems and the parameter programs give
the stated conclusions.
\end{proof}

\section{Exact verification}\label{sec:verification}

The verification code, recorded output, and instructions for reproducing
the checks are available in the accompanying GitHub repository.\footnote{\href{https://github.com/bogdan27182/pell-paper}{Pell verification code}.}

The scripts evaluate the displayed arithmetic subroutines on
exact integers and count their operations. Continued fractions supply
independent fundamental solutions, while multiplication by
$X_1+Y_1\sqrt d$ supplies the comparison sequences and coordinate sums.
The checks support the formulas on finite sets of inputs. The proofs
establish them over their stated domains.

The checks cover the moment formulas and their omitted tails, the signed
digit identity, valuation bounds on $256$ squares, and coordinate recovery
with exact binomial coefficients on $264$ squares. Parameter programs are
checked for $2\le d\le300$, and the continued-fraction bound is checked
for all $1956$ nonsquare $d\le2000$. The repository specifies the ranges for
each check and records the individual subroutine counts.

There are $30$ evaluations of the five complete Pell programs, including
squares with several nontrivial solutions. Their largest packed integer
has $96025942$ bits. These evaluations use exact binary digit counting
in place of the arithmetic HW expansion, which is checked separately on
small inputs because its intermediate integers are much larger. The
reported complete counts combine the operations outside HW with its
separately checked expansion. Small signed examples with $K<64$ use
widths satisfying the actual error condition, rather than the uniform
width theorem for $K\ge64$.

For the general-solution program in Section~\ref{sec:general}, the $54$
nonsquare values $d<200$ with $X_1\le50$ give $432$ evaluations at
$1\le n\le8$, each with $21$ additional operations. For the same values
of $d$, all $67240$ pairs $(d,b)$ with
$2X_1\le b<2X_1(X_1+1)-1$ fail to return both coordinates at $n=1$.

\section{Questions about intermediate size}\label{sec:questions}

The five constructions show how the choice of residual and weight affects
the number of operations and the length of the packing. They do not
establish an optimal operation count. Their differences in packed size
also leave the exponential height of the full bounds unchanged, because
Hamming weight dominates them all. The larger change is the choice of
parameter inside that subroutine, which trades a few operations for one
exponential level in Table~\ref{tab:parameters}.

Two size questions remain. First, the packings have exponentially many
bits in an a priori bound for $\log X_1$, whereas the output has
$O(\sqrt d\log d)$ bits. It is open here whether one can obtain the
needed coordinate sums by an arithmetic term with substantially smaller
intermediates. Second, even with a supplied small $e$, the HW program
forms an integer with $4M^2+1$ bits. This is polynomial in the value of
its input $M$, rather than in its bit length. A term for HW with
intermediate bit lengths polynomial in $\log M$ would change the size
bounds of all five constructions.

One formulation of the second question uses conversion of binary digits
to another base. Let $m=\sum_j\xi_j2^j$, with $\xi_j\in\{0,1\}$,
and choose an integer $\kappa\ge3$ with
$\kappa-1>\log_2m+1$. Put
$\mathcal D_\kappa(m)=\sum_j\xi_j\kappa^j$.
Direct summation of the contributions of each binary digit gives
\[
  (\kappa-2)\sum_{k\ge0}\fl{m/2^k}\kappa^k
    =\kappa\mathcal D_\kappa(m)-2m,\qquad
  \HW(m)=\mathcal D_\kappa(m)\md(\kappa-1).
\]
Thus an arithmetic term for this digit conversion, together with an
admissible choice of $\kappa$ and intermediate bit lengths polynomial
in $\log m$, would give such a Hamming-weight term.

\section{The general solution}\label{sec:general}

Once $(X_1,Y_1)$ is available, the recurrence
\[
  X_{n+2}=2X_1X_{n+1}-X_n,\qquad
  Y_{n+2}=2X_1Y_{n+1}-Y_n
\]
has rational generating functions
\[
  \mathrm{GF}_X(z)=\frac{1-X_1z}{1-2X_1z+z^2},\qquad
  \mathrm{GF}_Y(z)=\frac{Y_1z}{1-2X_1z+z^2}.
\]
Prunescu and Sauras-Altuzarra \cite[Cor.~37]{PSA25} obtain, for a
sufficiently large integer base $b$ and $n\ge1$,
\begin{align}
  X_n&=\fl{\frac{b^{n^2+2n}-X_1b^{n^2+n}}
                  {b^{2n}-2X_1b^n+1}}\md b^n,\label{eq:general-x}\\
  Y_n&=\fl{\frac{Y_1b^{n^2+n}}
                  {b^{2n}-2X_1b^n+1}}\md b^n.\label{eq:general-y}
\end{align}
We determine a common base from $X_1$ and prove its minimality.
A base is called \emph{admissible} here if the displayed denominators
are positive and both formulas give the required coordinates for all
$n\ge1$. This is the domain needed to interpret the formulas as
arithmetic terms on natural numbers.

\subsection{The least admissible common base}

The generating functions explain why a base can fail. Terms beyond the
coefficient being extracted contribute a positive tail. If that tail
reaches one, it changes the integer part before reduction modulo the
base. A tail smaller than one proves sufficiency, but necessity needs
more: a larger tail might contribute an entire multiple of the modulus
and disappear. We use both coordinates together to rule out this
possibility, exploiting $\gcd(X_1,Y_1)=1$.

\begin{lemma}\label{lem:tail}
Let $\Lambda=b^n>2X_1$, and define
\[
  T_n=\sum_{i\ge1}X_{n+i}\Lambda^{-i},\qquad
  U_n=\sum_{i\ge1}Y_{n+i}\Lambda^{-i}.
\]
The floor in \eqref{eq:general-x} equals
$\sum_{j=0}^n X_j\Lambda^{n-j}+\fl{T_n}$.
If $X_n<\Lambda$ and $T_n<1$, formula \eqref{eq:general-x}
returns $X_n$.
The corresponding statements hold for $Y_n$ and $U_n$.
\end{lemma}

\begin{proof}
Since $\eps<2X_1<\Lambda$, the generating-function series converge
at $1/\Lambda$. The fraction in \eqref{eq:general-x} is
\[
  \Lambda^n\mathrm{GF}_X(1/\Lambda)
   =\sum_{j=0}^nX_j\Lambda^{n-j}+T_n.
\]
The finite sum is an integer. If $T_n<1$, adding the tail does not
change its floor. Reduction modulo $\Lambda$ then returns $X_n$
when $X_n<\Lambda$.
The argument for $Y_n$ is the same.
\end{proof}

\begin{theorem}[Least admissible common base]\label{thm:general-base}
The least admissible common base for \eqref{eq:general-x} and
\eqref{eq:general-y} is
\begin{equation}\label{eq:base}
  b_0=2X_1(X_1+1)-1.
\end{equation}
Both formulas hold for all $n\ge1$ and all integers $b\ge b_0$.
For every integer $2X_1\le b<b_0$, the pair fails at $n=1$.
\end{theorem}

\begin{proof}
Write $a=X_1$ and $y=Y_1$ in this proof. Then $a\ge2$,
$1\le y<a$, and $\gcd(a,y)=1$.

First consider $b=b_0$ and $n=1$. Summing the generating-function
tail gives
\[
  T_1=\frac{b(2a^2-1)-a}{b^2-2ab+1}.
\]
For $b>2a$, the inequality $T_1<1$ is equivalent to
\begin{equation}\label{eq:base-quadratic}
  f(b)=b^2-(2a^2+2a-1)b+(a+1)>0.
\end{equation}
At $b_0$, one has $f(b_0)=a+1>0$.
Since $a<b_0$ and $U_1<T_1$, Lemma~\ref{lem:tail} proves both
formulas at $n=1$.

For $n\ge2$, the recurrence implies $X_m<(2a)^m$ for $m\ge1$.
The inequality
\[
  b_0^2-(2a)^3-2a=4a^4-6a+1>0
\]
starts an induction giving
\[
  b_0^n>(2a)^{n+1}+2a\qquad(n\ge2).
\]
To pass to the next value of $n$, multiply by $b_0>2a$.
Consequently
\[
  T_n<\sum_{i\ge1}(2a)^{n+i}b_0^{-ni}
      =\frac{(2a)^{n+1}}{b_0^n-2a}<1.
\]
Also $X_n<(2a)^n<b_0^n$ and $U_n<T_n$.
Lemma~\ref{lem:tail} proves both extraction formulas.
Increasing $b$ decreases the tails and increases the modulus, proving
sufficiency for $b\ge b_0$.

It remains to exclude smaller bases. At $b=2a$, the denominator for
$n=1$ is one, and the $X$-formula returns zero modulo $b$.
Suppose now $b\ge2a+1$, and set $s_b=b-2a$ and $D_b=bs_b+1$.
The two tails at $n=1$ are
\[
  T=\frac{b(2a^2-1)-a}{D_b},\qquad
  U=\frac{y(2ab-1)}{D_b}.
\]
The floors before reduction modulo $b$ are
$b+a+\fl T$ and $y+\fl U$.
If both residues are correct, then
$\fl T=kb$ and $\fl U=lb$ for nonnegative integers $k,l$.
Put $H_b=2a^2-1-kD_b$ and $J_b=2ay-lD_b$.
The floor inequalities give
\[
  \frac ab\le H_b<s_b+\frac{a+1}{b},\qquad
  \frac yb\le J_b<s_b+\frac{y+1}{b}.
\]
Thus $1\le H_b,J_b\le s_b$.
On the other hand,
\[
  yH_b-aJ_b+y=(al-yk)D_b,
\]
and the left side lies strictly between $-D_b$ and $D_b$, because
\[
  -D_b<2y-as_b\le yH_b-aJ_b+y
       \le y(s_b+1)-a<as_b<D_b.
\]
Hence $al=yk$. Coprimality gives $a\mid k$, whereas
\[
  0\le k\le T/b<2a^2/D_b<a
\]
because $D_b>2a$. Therefore $k=l=0$, and simultaneous correctness
requires $T<1$.

The quadratic in \eqref{eq:base-quadratic} has one root in $(0,1)$
and the other in $(b_0-1,b_0)$: indeed
\[
  f(0)=a+1>0,\qquad
  f(1)=f(b_0-1)=3-a-2a^2<0,\qquad f(b_0)=a+1>0.
\]
For integers $b>2a$, the condition $T<1$ is therefore equivalent
to $b\ge b_0$.
For the remaining bases $2\le b<2a$, the denominator at $n=1$ is
$b(b-2a)+1<0$, so these bases are outside the stated domain.
\end{proof}

For example, $d=7$ has $(X_1,Y_1)=(8,3)$, giving $b_0=143$.
The two formulas at $n=1$ return $(8,3)$ at base $143$.
At base $142$, the first returns $9$ because its tail exceeds one.
The separate coordinate bases found in \cite[Example~38]{PSA25}
are $143$ for $X$ and $64$ for $Y$.

\subsection{The straight-line program in \texorpdfstring{$d,n$}{d,n}}

Choose any of the five complete fundamental-solution programs of
Section~\ref{sec:parameters}. Its output $(X_1,Y_1)$ is the input
to the following program, together with $n\ge1$.

\textbf{G1. Extraction base.}
\[
  h_b=X_1+1,\qquad p_b=X_1h_b,\qquad
  r_b=2p_b,\qquad b=r_b\tsub1.
\]
\emph{Operation count:} $4$.

\textbf{G2. Powers.}
\begin{align*}
  \Lambda&=b^n,& \Lambda_2&=\Lambda\Lambda,& \ell_n&=nn+n,\\
  \Phi&=b^{\ell_n},& \Psi&=\Phi\Lambda.&&
\end{align*}
\emph{Operation count:} $6$.

\textbf{G3. Common denominator.}
\[
  \Omega=X_1\Lambda,\qquad
  \Xi=(\Lambda_2\tsub2\Omega)+1.
\]
\emph{Operation count:} $4$.

\textbf{G4. First coordinate.}
\[
  X_n=\fl{(\Psi\tsub X_1\Phi)/\Xi}\md\Lambda.
\]
\emph{Operation count:} $4$.

\textbf{G5. Second coordinate.}
\[
  Y_n=\fl{Y_1\Phi/\Xi}\md\Lambda.
\]
\emph{Operation count:} $3$.

\emph{Total additional operation count:} $4+6+4+4+3=21$.
All subtractions are exact because $b^n>2X_1$.
Theorem~\ref{thm:general-base} proves the output formulas.
Composing with Program SC gives $115$ operations with elementary
parameters and $124$ with Hua parameters, including construction of
the parameters. Using the HW-C substitution instead gives $119$ and
$128$ operations, respectively.

The largest integer in G1--G5 is $\Psi=b_0^{n^2+2n}$.
Its bit length is $O(n^2\log X_1)$.
Hua's bound gives $O(n^2\sqrt d\log d)$ as a bound in $d,n$.
The additional program therefore has polynomial intermediate bit length
in $n$ and in the length of the fundamental solution.

%%%%%%%%%%%%%%%%%%%%%%%%%%%%%%%%%%%%%%%%%%%%%%%%%%%%%%%%%%%%%%%%%%%%%%%%

\section*{Acknowledgments}

This paper was produced in an AI-assisted workflow. AI assistants helped
explore candidate constructions, draft proofs, and write software for numerical experiments.

\end{document}